\documentclass[11 pt]{amsart}
\usepackage[T1]{fontenc}
\usepackage[utf8]{inputenc}
\usepackage[english]{babel}
\usepackage{mathtools}
\usepackage{amsfonts}
\usepackage{amsthm}
\usepackage{braket}
\usepackage{lipsum}
\newcommand{\N}{\mathbb{N}}
\newcommand{\h}{\mathrm{H}}
\newcommand{\C}{\mathbb{C}}
\newcommand{\D}{\mathbb{D}}
\newcommand{\spa}{\mathrm{span}}
\newcommand{\dist}{\mathrm{dist}}
\newcommand{\diag}{\mathrm{diag}}
\newcommand{\rkhs}{\mathcal{H}}
\numberwithin{equation}{section}
\theoremstyle{plain}
\newtheorem*{theo*}{Theorem}
\newtheorem{theo}{Theorem}[section]
\newtheorem{prop}[theo]{Proposition}
\newtheorem{lemma}[theo]{Lemma}
\newtheorem{question}[theo]{Question}
\theoremstyle{definition}
\newtheorem{defi}[theo]{Definition}
\newtheorem{rem}[theo]{Remark}
\newtheorem{ex}[theo]{Example}
\title{Weakly Separated Bessel Systems of Model Spaces}
\author{Alberto Dayan}
\address{Department of Mathematics\newline Washington University in St. Louis,\newline One Brookings Drive, St. Louis, MO 63130, USA}
\email{alberto.dayan@wustl.edu}
\date{\today}
\thanks{The author was partially supported by National Science Foundation Grant DMS 1565243}
\begin{document}
\maketitle
\begin{abstract}
We show that any weakly separated Bessel system of model spaces in the Hardy space on the unit disc is a Riesz system and we highlight some applications to interpolating sequences of matrices. This will be done without using the recent solution of the Feichtinger conjecture, whose natural generalization to multi-dimensional model sub-spaces of $\h^2$ turns out to be false.
\end{abstract}
\section{Introduction}
Let $\h^2$ be the Hardy space on the unit disc $\D$, that is, the reproducing kernel Hilbert space of those power series centered at the origin with square-summable Taylor coefficients. Its kernel $s$ is the well-studied \emph{Szegö kernel}
\[
s_w(z):=\frac{1}{1-\overline{w}z}\qquad w, z\in\D
\]
and its multiplier algebra can be identified isometrically with $\h^\infty$, the algebra of bounded holomorphic functions on $\D$. A key role for the study of the function theory and the hyperbolic geometry of the unit disc is played by \emph{inner functions}, which are those bounded analytic functions on the unit disc with an unimodular radial limit almost everywhere on the unit circle. Given an inner function $\Theta$, one can define the associated \emph{model space}
\[
H_\Theta:=\h^2\ominus\Theta\h^2
\]
as the orthogonal complement in the Hardy space of all multiples of $\Theta$ in $\h^2$. A great treatment of the main properties of model spaces, together with their interactions with operator theory on spaces of analytic functions, can be found in \cite{ross}.\\
Any function in $\h^2$ that vanishes with multiplicity $m$ at a point $\lambda$ in $\D$ is divisible in $\h^2$ by a \emph{Blaschke factor}, i.e., an inner function of the form $b_\lambda^m$, where
\[
b_\lambda(z):=\frac{\lambda-z}{1-\overline{\lambda}z}\qquad z\in\D.
\]
Therefore the model space associated to $b_\lambda^m$ is $m$-dimensional and it is spanned by the kernels at $\lambda$ that represent up to $m-1$ derivatives of any $\h^2$ function at $\lambda$, that is,
\begin{equation}
\label{enq:lambdamodel}
H_{b_\lambda^m}=\spa\left\{s_\lambda, \frac{\partial~s_\lambda}{\partial\overline{w}},\dots, \frac{\partial^{m-1}~s_\lambda}{\partial\overline{w}^{m-1}}\right\}.
\end{equation}
Since a model space is a subspace of $\h^2$ generated by an inner function, it comes natural to ask whether function theoretical properties of a sequence of inner functions $(\Theta_n)_{n\in\N}$ translate to Euclidean properties for the sequence $(H_{\Theta_n})_{n\in\N}$. Out of the many results that constitute such a valuable dictionary between operator theory and function theory, one of the most significant for the purposes of this note can  be found in \cite[Th. 3.2.14]{nik}:
\begin{theo}
\label{theo:nik}
Let $(\Theta_n)_{n\in\N}$ be a sequence of inner functions such that $\Theta:=\prod_{n=1}^\infty\Theta_n$ converges uniformly on any compact subset of $\D$. The following are equivalent:
\begin{description}
\item[(i)] For any bounded sequence $(\phi_n)_{n\in\N}$ in $\h^\infty$ there exists a function $\phi$ in $\h^\infty$ such that
\begin{equation}
\label{eqn:nik:interp}
\phi-\phi_n\in\Theta_n\h^2,\qquad n\in\N;
\end{equation}
\item[(ii)] There exists a $C\geq1$ such that, for any sequence $(h_n)_{n\in\N}$ of unit vectors in $\h^2$ such that $h_n$ belongs to $H_{\Theta_n}$ for any $n$ in $\N$ and for any $(a_n)_{n\in\N}$ in $l^2$,
\begin{equation}
\label{eqn:nik:riesz}
\frac{1}{C^2}~\sum_{n\in\N}|a_n|^2\leq\left|\left|\sum_{n\in\N} a_nh_n\right|\right|^2\leq C^2~\sum_{n\in\N}|a_n|^2;
\end{equation}
\item[(iii)] There exists a positive $\delta$ such that, for any $z$ in $\D$,
\[
|\Theta(z)|\geq\delta~\inf_{n\in\N}|\Theta_n(z)|.
\]
\end{description}
\end{theo}
Note that \eqref{eqn:nik:interp} is stated in \cite[Th. 3.2.14]{nik} as 
\[
\phi-\phi_n\in\Theta_n\h^\infty,\qquad n\in\N.
\]
This is equivalent to \eqref{eqn:nik:interp}, since $\h^\infty$ is contained in $\h^2$ and since, if for any $n$ there exists a $g$ in $\h^2$ such that
\[
\theta_n~g=\phi-\phi_n\in\h^\infty,
\]
then the radial limit of $|g|$ on the unit circle must be bounded by $|\phi-\phi_n|$, since $\Theta_n$ is inner, and therefore $g$ is actually in $\h^\infty$.
A sequence of closed subspaces $(H_n)_{n\in\N}$ of a Hilbert space $\mathcal{H}$ that satisfies (ii) is called a {\bf Riesz system}, and the least $C$ for which \eqref{eqn:nik:riesz} holds is the \emph{Riesz bound} of the sequence. If the least $C$ such that the right hand side of \eqref{eqn:nik:riesz} holds is finite then $(H_n)_{n\in\N}$ is a {\bf Bessel systems} with \emph{Bessel bound} $C$. On the other hand, condition (iii) is equivalent to
\begin{equation}
\label{eqn:nik:gc}
\sup_{n\in\N}\prod_{j\ne n}|\Theta_j(z)|\geq\delta, 
\end{equation}
to hold uniformly in $z$. Indeed, condition (iii) above states that there exists a positive $\delta$ such that, for any $z$ in $\D$ and for any $0<\varepsilon<\delta$ there exists a $n$ in $\N$ such that 
\[
|\Theta_n(z)|\leq\frac{|\Theta(z)|}{\delta-\varepsilon},
\]
or, equivalently, such that
\[
\prod_{j\ne n}|\Theta_j(z)|\geq \delta-\varepsilon,
\]
which yields \eqref{eqn:nik:gc} (in case $\Theta_n(z)=0$, take a limit for $w$ going to $z$ and then use continuity).
\\
Moreover, condition (iii) is related in \cite[Lec. IX]{treatise} to separation conditions on the subspaces in $(H_{\Theta_n})_{n\in\N}$, being equivalent to asserting that the sine of the angle between any model space and the closure of the linear span of all the others is uniformly bounded below:
\begin{equation}
\label{eqn:strongmodel}
\inf_{n\in\N}\sin\left(H_{\Theta_n}, \overline{\underset{j\ne n}{\spa}}\{H_{\Theta_j}\}\right)>0.
\end{equation}
Here the sine between two closed sub-spaces $K_1$ and $K_2$ of a Hilbert space $\rkhs$ is the least sine of the angle between two vectors chosen in $K_1$ and $K_2$ or, equivalently, $$\sin(K_1, K_2)=||T||^{-1},$$ where
\[
\begin{split}
T\colon\spa\{K_1, K_2\}&\mapsto\spa\{K_1, K_2\}\\
T_{|K_1}=Id_{K_1}&\qquad T_{|K_2}=0.
\end{split}
\]
If \eqref{eqn:strongmodel} holds, we say that $(H_{\Theta_n})_{n\in\N}$ is {\bf strongly separated}, whereas {\bf weak separation} will correspond to an uniform bound from below for the angle between any pair of distinct model spaces:
\begin{equation}
\label{eqn:weakmodel}
\inf_{n\ne j}\sin(H_{\Theta_n}, H_{\Theta_j})>0.
\end{equation}
Suppose now that $\Theta_n=b_{\lambda_n}^{m_n}$, for some sequence $(\lambda_n)_{n\in\N}$ in the unit disc and some sequence $(m_n)_{n\in\N}$ of positive integers. Condition (i) of Theorem \ref{theo:nik} becomes then an \emph{interpolation property}: for any bounded sequence $(\phi_n)_{n\in\N}$ in $\h^\infty$ there exists a bounded analytic function $\phi$ that, for any $n$ in $\N$, agrees with $\phi_n$ at $\lambda_n$ up to its $m_n-1^\text{st}$ derivative. Theorem \ref{theo:nik} is therefore a great example of how the deep interconnection between operator theory and function theory greatly helps the studying and the understanding of  \emph{interpolating sequences}.
\\

 Let $\mathcal{H}_k$ be a reproducing kernel Hilbert space of analytic functions on a domain $D$ of $\C^d$, and let $\mathcal{M}_k$ be its multiplier algebra. A sequence $Z=(z_n)_{n\in\N}$ in $D$ is {\bf interpolating} for $\mathcal{M}_k$ if for any bounded sequence $(w_n)_{n\in\N}$ in $\C$ there exists a function $\phi$ in $\mathcal{M}_k$ such that $\phi(z_n)=w_n$ for any $n$ in $\N$. Intuitively,  an interpolating sequence is a separated sequence, as we need to be able to specify the values of $\phi$ \emph{arbitrarily} at the nodes $(z_n)_{n\in\N}$. It is also not surprising that the separation conditions we will look at depend on the kernel $k$: $Z$ is a {\bf weakly separated} sequence if there exists a positive $M$ such that, for any $n\ne j$ in $\N$, there exists a function $\phi_{n, j}$ whose norm in $\mathcal{M}_k$ doesn't exceed $M$ and that separates $z_n$ and $z_j$, that is, 
\[
\phi_{n, j}(z_n)=1\qquad\phi_{n, j}(z_j)=0.
\]
A celebrated work of Carleson,\cite{carlocm}, characterized interpolating sequences for $\h^\infty$:
\begin{theo}
\label{theo:carlo}
A sequence $\Lambda=(\lambda_n)_{n\in\N}$ is interpolating for $\h^\infty$ if and only if it is weakly separated and the measure
\[
\mu_\Lambda:=\sum_{n\in\N}(1-|\lambda_n|^2)\delta_{\lambda_n}
\]
satisfies the embedding condition
\begin{equation}
\label{eqn:cm}
||f||_{L^2(\D, \mu_\Lambda)}\leq C_\Lambda~||f||_{\h^2}\qquad f\in\h^2.
\end{equation}
\end{theo}
A measure $\mu$ on a domain $D$ that embeds continuously a reproducing kernel Hilbert space $\mathcal{H}_k$ on $D$ into $L^2(D, \mu)$  is called a {\bf Carleson measure} for $\mathcal{H}_k$. One can find in \cite{carlocm} a characterization of Carleson measures for $\h^2$ that involves a one-box condition, and hence the hyperbolic geometry of the unit disc. It turns out that \eqref{eqn:cm} holds if and only if the sequence of all lines through the kernels $(s_{\lambda_n})_{n\in\N}$ (that is, the sequence one dimensional sub-spaces spanned by the kernels $s_{\lambda_n}$) is a Bessel system,\cite[Prop. 9.5]{john}, highlighting once again a correspondence between the hyperbolic geometry of the unit disc and the Euclidean geometry of the Hardy space. Moreover, such a characterization of Carleson measures in terms of Bessel systems allows to extend Theorem \ref{theo:carlo} to some multiplier algebra other than $\h^\infty$. For example, a class of multiplier algebras for which an analogous of Theorem \ref{theo:carlo} holds is the one associated with \emph{complete Pick kernels}. One of the most important properties that connects interpolating sequence to the study of related Hilbert spaces is the fact that, for any multiplier $\phi$ in $\mathcal{M}_k$, any kernel function $k_z$ in $\mathcal{H}_k$ is an eigenfunction of the adjoint of the multiplication operator $M_\phi$:
\begin{equation}
\label{eqn:eigen}
M^*_\phi(k_z)=\overline{\phi(z)}~k_z\qquad z\in D,
\end{equation}
as a straightforward computation using the property of adjoints shows. In particular, if $M_\phi$ is a contraction and $\phi(z_n)=w_n$ for any $n$ in $\N$, then the linear map $T$ from $S_Z:=\overline{\underset{n\in\N}{\spa}}\{k_{z_n}\}$ to itself given by
\begin{equation}
\label{eqn:contr}
T(k_{z_n}):=\overline{w_n}~k_{z_n}\qquad n\in\N
\end{equation}
is a contraction. A reproducing kernel Hilbert space is said to have the {\bf Pick property} if the existence of such a contraction $T$ is also a sufficient condition for the existence of a function $\phi$ in the unit ball of $\mathcal{M}_k$ such that $\phi(z_n)=w_n$. In particular, this says that $M^*_\phi$ is an extension of $T$ to $\mathcal{H}_k$ and $||M^*_\phi||=||T||$, which implies that any two disjoint sets of points $Z_1$ and $Z_2$ can be separated by a function in $\mathcal{M}_k$ of norm at most $M$ if and only if the angle between $S_{Z_1}$ and $S_{Z_2}$ in $\mathcal{H}_k$ is bounded below by $1/M$:
\begin{equation}
\label{eqn:anglespick}
\sup\left\{||\phi||_{\mathcal{M}_k}\,\big|\, \phi_{|Z_1}=1,\, \phi_{|Z_2}=0\right\}=\frac{1}{\sin(S_{Z_1}, S_{Z_2})}.
\end{equation}
Moreover,\cite[Th. 9.19]{john}, this implies that $Z$ is interpolating if and only if the sequence of lines through the kernels $(k_{z_n})_{n\in\N}$ is a Riesz system.\\
Since \eqref{eqn:contr} being a contraction is equivalent to the infinite matrix
\[
(1-\overline{w_n}w_j)k_{z_j}(z_n)\qquad n, j\in\N
\]
being positive semi-definite (i.e., all its finite principal minors are positive semi-definite), one can extend the Pick property to the case of \emph{matrix-valued} functions in $\mathcal{H}_k$, by defining $\mathcal{H}_k$ to have the \emph{$s\times t$ Pick property} if whenever $z_1, \dots, z_N$ are points in $D$ and $W_1, \dots, W_N$ are $s\times t$ matrices such that
\[
(Id-W_n^*W_j)k_{z_j}(z_n)\geq0
\]
then there exists a multiplier $\phi$ in the closed unit ball of 
\[
\mathcal{M}(\mathcal{H}_k\otimes\C^t, \mathcal{H}_k\otimes\C^s):=\left\{{\bf\phi}=(\phi_{l, r})\,|\, l=1, \dots, s,\, r=1,\dots t,\, \sup_{{\bf f}\ne0}\frac{||\phi{\bf f}||_{\mathcal{H}_k\otimes\C^s}}{||{\bf f}||_{\mathcal{H}_k\otimes\C^t}}<\infty\right\}
\]
such that $\phi(z_i)=W_i$, for $i=1, \dots, N$. We say that $\mathcal{H}_k$ has the {\bf complete Pick property} if it has the $s\times t$ Pick property for any positive integers $s$ and $t$. The Hardy space $\h^2$ has the complete Pick property, as well as some of its natural generalizations, such as the reproducing kernel Hilbert spaces $\mathcal{H}_s$ on $\D$, $-1\leq s\leq0$, defined by the kernels
\[
k^s_w(z):=\sum_{n=0}^\infty (n+1)^s(\overline{w}z)^n\qquad z, w\in\D
\] 
and the \emph{Drury-Arveson space} $\h^2_d$ on the $d$-dimensional unit ball $\mathbb{B}_d$ defined by the kernel
\[
b_w(z):=\frac{1}{1-\braket{z, w}},\qquad z, w\in\mathbb{B}_d.
\]
For instance, see \cite[Ch. 7]{john}. In a recent work, \cite{pk}, Aleman, Hartz, M$^{\text{c}}$Carthy and Richter extended Theorem \ref{theo:carlo} by showing that any weakly separated sequence $Z$ on a domain $D$ such that the sequence of lines through the kernels $(k_{z_n})_{n\in\N}$ is a Bessel system is an interpolating sequence for $\mathcal{M}_k$, provided that $\mathcal{H}_k$ has the complete Pick property. This is done by using the recent positive answer to the \emph{Feichtinger conjecture}, which states that any Bessel system of one-dimensional subspaces is the disjoint union of finitely many Riesz systems.  It has been shown, \cite{casazza} \cite{weaver}, that the Feichtinger conjecture is equivalent to many other conjectures in operator theory, including the Paving conjecture, who had been proved by the well-known work of Marcus, Spielmann and Srivastava \cite{marcus}.\\

In \cite{intmat}, the author asked whether the positive answer to the Feichtinger conjecture can be extended to multi-dimensional model spaces of $\h^2$, that is, if any Bessel system of model spaces is the disjoint union of finitely many Riesz systems. We show in Section \ref{sec:feicmodel} that this is not the case, though any Bessel systems of model spaces satisfying \eqref{eqn:weakmodel} is in fact a Riesz system:
\begin{theo}
\label{theo:main}
Any weakly separated Bessel system of model spaces in $\h^2$ is a Riesz system.
\end{theo}\,\\

The motivation for Theorem \ref{theo:main} is the study of \emph{interpolating sequences of matrices} introduced by the author in \cite{intmat}. As Section \ref{sec:intmat} explains in details, we say that a sequence of square matrices $A=(A_n)_{n\in\N}$ with eigenvalues in $\D$ is interpolating if, for any bounded sequence $(w_n)_{n\in\N}$ in $\C$, there exists a bounded holomophic function $\phi$ such that
\[
\phi(A_n)=w_n~Id.
\]
 If $P_n$ is, for any $n$ in $\N$, the minimal polynomial of $A_n$, then
\begin{equation}
\label{eqn:Bn}
B_n(z):=\frac{P_n(z)}{\overline{P_n}\left(\frac{1}{\overline{z}}\right)}\qquad z\in\D
\end{equation}
is a Blaschke product with zeros at the eigenvalues of $A$, and any $\h^2$ function that vanishes at $A_n$ is a multiple of $B_n$. Let $H=(H_n)_{n\in\N}$ be the sequence of model spaces associated to $(B_n)_{n\in\N}$. The author extended  in \cite[Th. 6.6]{intmat} Theorem \ref{theo:carlo} to sequences of matrices of uniformly bounded dimensions, by showing that $A$ is interpolating  if and only if $H$ is a weakly separated Bessel system. Theorem \ref{theo:main} can be rephrased to drop the extra assumption on the sizes of the matrices in $A$:
\begin{theo}
\label{theo:carlomat}
$A$ is interpolating if and only if the sequence $H$ is a weakly separated Bessel system. 
\end{theo}
Section \ref{sec:proof} deals with the proof of Theorem \ref{theo:main}. Section \ref{sec:intmat} provides a brief summary of the content of \cite{intmat} and gives an argument for Theorem \ref{theo:carlomat}, together with an explicit example of an interpolating sequence of matrices. We also give in Section \ref{sec:feicmodel} an example of a sequence of matrices whose associated sequence of model spaces $(H_n)_{n\in\N}$ is a Bessel system which can not be written as the disjoint union of finitely many Riesz systems, giving a negative answer to a question posed by the author in \cite{intmat}.\\

The author would like to thank John M$^{\text{c}}$Carthy for the valuable suggestions given during all the conversations that led to this work. The author is also indebted to the reviewer for some extremely important remarks and suggestions.
\section{The Proof of the Main Result}
\label{sec:proof}
 The first main tool for the proof of Theorem \ref{theo:main} can be found in \cite[Lec. IX]{treatise}, and relates the sine of the angle between two model spaces $H_{\Theta_1}$ and $H_{\Theta_2}$ with the constant in Carleson corona Theorem:
\begin{theo}
\label{theo:treatise}
There exists a constant $c\geq1$ such that, for any $\Theta_1$ and $\Theta_2$ inner functions on $\D$ such that 
\[
\inf_{z\in\D}\max\{|\Theta_1(z)|, |\Theta_2(z)|\}=\delta\geq0
\]
then
\[
\frac{\delta^3}{c}\leq\sin(H_{\Theta_1}, H_{\Theta_2})\leq c\delta.
\]
\end{theo}
We are also going to use the following re-statement of the Bessel system condition:
\begin{prop}
\label{prop:bessel}
A sequence $(H_n)_{n\in\N}$ of closed sub-spaces of a Hilbert space $\mathcal{H}$ is a Bessel system with Bessel bound $M$ if and only if, for any sequence $(h_n)_{n\in\N}$ of unit vectors such that $h_n$ belongs to $H_n$ for any $n$ in $\N$,
\begin{equation}
\label{eqn:bessel:sup}
\sup_{||x||=1}\sum_{n\in\N}|\braket{x, h_n}|^2=M^2.
\end{equation}
\end{prop}
\begin{proof}
The idea of the proof comes from \cite[Prop. 9.5]{john}. Choose for any $n$ in $\N$ a unit vector $h_n$ in $H_n$, and suppose first that \eqref{eqn:bessel:sup} holds. Then, for any finitely supported $(a_n)_{n\in\N}$,
\[
\begin{split}
\left|\left|\sum_{n\in\N} a_nh_n\right|\right|^2&=\sup_{||x||=1}\left|\braket{x, \sum_{n\in\N} a_nh_n }\right|^2\\
&=\sup_{||x||=1}\left|\sum_{n\in\N}\braket{x, h_n}~\overline{a_n}\right|^2\\
&\leq M^2 \sum_{n\in\N}|a_n|^2,
\end{split}
\]
thanks to \eqref{eqn:bessel:sup} and Cauchy-Schwartz's inequality. The case of infinitely supported sequences $(a_n)_{n\in\N}$ in $l^2$ follows by considering the limit case of finitely supported sequences.\\
Conversely,  let $M$ be the Bessel bound for the sequence $(H_n)_{n\in\N}$, and fix a unit vector $x$ in $\mathcal{H}$ . Then set $a_n=\braket{x, h_n}$, and observe that
\[
\begin{split}
\sum_{n\in\N}|\braket{x, h_n}|^2=&\braket{x, \sum_{n\in\N}a_n h_n}\\
\leq&\left|\left|\sum_{n\in\N} a_nh_n\right|\right|\\
\leq&M\left(\sum_{n\in\N}|a_n|^2\right)^\frac{1}{2}\\
=&M\left(\sum_{n\in\N}|\braket{x, h_n}|^2\right)^\frac{1}{2}.
\end{split}
\]
By truncating the sum $\underset{n\in\N}\sum|\braket{x, h_n}|^2$, one sees that this implies that $(a_n)_{n\in\N}$ must be square summable, and this concludes the proof.
\end{proof}
\begin{rem}
\label{rem:proj}
Fixed $x$ in $\mathcal{H}$, we can choose $(h_n)_{n\in\N}$ so that the sum in \eqref{eqn:bessel:sup} attains its maximum. We can actually maximizes each term of the sum, by setting $h_n$ to be the orthogonal projection onto $H_n$ of $x$, divided by its norm:
\[
f_n(x):=\frac{P_{H_n}(x)}{\left|\left|P_{H_n}(x)\right|\right|}.
\] 
Proposition \ref{prop:bessel} then says that $(H_n)_{n\in\N}$ has a finite Bessel bound $M$ if and only if 
\begin{equation}
\label{eqn:bessel:proj}
\sup_{||x||=1}\sum_{n\in\N}(1-\dist_{\h^2}^2(x, H_n))=\sup_{||x||=1}\sum_{n\in\N}|\braket{x, f_n(x)}|^2= M^2
\end{equation}
\end{rem}
Lastly, we are going to use the one dimensional case of \cite[Th. 5.1]{intmat}. For any $x$ in a Hilbert space $\mathcal{H}$ let $\hat{x}$ denote its normalization $x/||x||$. 
\begin{theo}
\label{theo:distinner}
For any inner function $\Theta$ on $\D$,
\[
\dist_{\h^2}(\hat{s_z}, H_\Theta)=|\Theta(z)|\qquad z\in\D.
\]
\end{theo}
\begin{proof}
Since the orthogonal projection onto $\Theta\h^2=\h^2\ominus H_\Theta$ is $M_\Theta M_{\Theta}^*$ one gets
\[
\dist_{\h^2}(\hat{s}_z, H_\Theta)=||M_\Theta M^*_\Theta(\hat{s}_z)||=||M^*_\Theta(\hat{s}_z)||=|\Theta(z)|.
\]
\end{proof}
We are now ready to prove Theorem \ref{theo:main}:
\begin{proof}[Proof of Theorem \ref{theo:main}]
Let $(H_{\Theta_n})_{n\in\N}$ be a weakly separated Bessel system. We will show that \eqref{eqn:nik:gc} holds, and Theorem \ref{theo:nik} will conclude the proof. Thanks to \eqref{eqn:bessel:proj}, for any fixed $z$ in $\D$ there exists a positive integer $n_z$ that minimizes the distance between $\hat{s_z}$ and $H_n$:
\[
\dist_{\h^2}(\hat{s_z}, H_{n_z})=\min_{n\in\N}\dist_{\h^2}(\hat{s_z}, H_n),
\]
which thanks to Theorem \ref{theo:distinner} becomes 
\[
|\Theta_{n_z}(z)|=\min_{n\in\N}|\Theta_n(z)|.
\]
Therefore, by Theorem \ref{theo:treatise} and weak separation we have that 
\[
\inf_{z\in\D}\inf_{n\ne n_z}|\Theta_n(z)|>0,
\]
which implies that 
\[
\prod_{n \ne n_z}|\Theta_n(z)|=\sup_{n\in\N}\prod_{j\ne n}|\Theta_j(z)|
\]
is bounded below uniformly on $z$ if and only if 
\[
\sum_{n\ne n_z}(1-|\Theta_n(z)|^2)=\sum_{n\ne n_z} (1-\dist_{\h^2}^2(\hat{s_z}, H_{\Theta_n}))
\]
is uniformly bounded on $z$, which is true thanks to Remark \ref{rem:proj}.
\end{proof}
Observe that the proof of Theorem \ref{theo:main} used a weaker version of the Bessel system condition, in which the sup in \eqref{eqn:bessel:sup} is taken only on normalized kernel functions, rather than on all unit vectors in $\h^2$. It remains open for us whether such a weaker condition is enough to characterize Bessel systems of model spaces, and a positive answer for the special case in which each $\Theta_n$ is a Blaschke product would be of great interest for us, as this is the case that we consider when we apply Theorem \ref{theo:main} to interpolating sequences of matrices, as we will see in Section \ref{sec:intmat}:
\begin{question}
\label{q:bessel}
Is any sequence of model spaces $(H_{\Theta_n})_{n\in\N}$ in $\h^2$ satisfying
\begin{equation}
\label{eqn:q:bessel}
\sup_{z\in\D}\sum_{n\in\N}1-|\Theta_n(z)|^2<\infty
\end{equation}
a Bessel system? Is it true if $\Theta_n$ is, for any positive integer $n$, a Blaschke product?
\end{question}
\begin{rem}
\label{rem:bessel:lines}
Question \ref{q:bessel} has a positive answer whenever $\Theta_n=b_{\lambda_n}$ is, for any positive integer $n$, a degree-one Blaschke factor at a point $\lambda_n$, and therefore whenever $H_{\Theta_n}$ is the line spanned by the Szegö kernel at $\lambda_n$, \cite[Ch. VI, Lemma 3.3]{gar}.
\end{rem}
\section{Interpolating Matrices}
\label{sec:intmat}
The motivation for Theorem \ref{theo:main} is the study of \emph{interpolating sequences of matrices}. The author asked in \cite{intmat} whether some well known characterizations for interpolating sequences for $\h^\infty$ such as Theorem \ref{theo:carlo} extend to sequences of square matrices $A=(A_n)_{n\in\N}$, without assuming any restriction on the sequence of their dimensions. The fact that a square matrix might have a non trivial algebraic structure invariant under holomorphic functions (its eigenspaces, for example), makes an interpolation problem using matrices a bit trickier than the classic one: given two points $\lambda$ and $w$ in $\D$ there is no function $\phi$ in $\h^\infty$ that maps
\[
A=
\begin{bmatrix}
\lambda & 0\\
0 & \lambda
\end{bmatrix}
\]
to
\[
W=
\begin{bmatrix}
w & 1\\
0 & w
\end{bmatrix},
\]
although both $A$ and $W$ are bounded in the operator norm and the constant function $w$ is a contraction in $\h^\infty$ that sends the spectrum of $A$ to the spectrum of $W$. Since choosing bounded targets in the operator norm makes even a one-point interpolation problem impossible to solve via bounded analytic functions, in order to define interpolating matrices one has to identify a target with a bounded sequence in $\h^\infty$, \cite[Def. 1.1]{intmat}:
\begin{defi}[Interpolating Matrices]
\label{defi:intmat}
A sequence $A=(A_n)_{n\in\N}$ of square matrices with spectra in the open unit disc  is interpolating for $\h^\infty$ if for any bounded sequence $(\phi_n)_{n\in\N}$ in $\h^\infty$ there exists a $\phi$ in $\h^\infty$ such that
\[
\phi(A_n)=\phi_n(A_n),\qquad n\in\N.
\]
\end{defi}
Equivalently, \cite[Th. 4.1]{intmat}, one can choose diagonal targets, and define $A$ to be interpolating if for any bounded sequence $(w_n)_{n\in\N}$ in $\C$ there exists a bounded analytic function $\phi$ such that
\[
\phi(A_n)=w_n~Id,\qquad n\in\N.
\]
Here an analytic function on the unit disc is applied to a square matrix via the Riesz-Dunford functional calculus, hence the assumption on the spectra of the matrices in $A$. In order to characterize interpolating sequences of matrices, a rather trivial yet important observation is that, for any pair of similar matrices $M$ and $N$ with spectra in $\D$ and for any holomorphic function $f$ on the unit disc then $f(M)$ and $f(N)$ are similar as well, and the matrix that performs both similarities is the same,
\[
M=P^{-1}NP\implies f(M)=P^{-1}f(N)P,
\]
as an elementary computation using the power series of $f$ shows. As a consequence, we can assume without loss of generality that each matrix of the sequence $A$ is in its \emph{Jordan canonical form}: if, for any positive integer $n$, $\lambda_{n, 1}, \dots, \lambda_{n, k_n}$ are the eigenvalues of $A_n$, then
\[
A_n=\diag(J_{n, 1}, \dots, J_{n, k_n}),
\]
where $J_{n, j}$ is a Jordan block of size $m_{n, j}$
\[
J_{n,j}=\begin{bmatrix}
\lambda_{n,j} & 1 & 0 & \dots & 0\\
0 & \lambda_{n, j} & 1 & \dots & 0\\
\vdots & \vdots & \vdots & \vdots & \vdots\\
0 & 0 & \dots & \lambda_{n, j} & 1\\
0 & 0 & 0 & 0 & \lambda_{n, j}
\end{bmatrix}.
\]
Since, for any function $f$ holomorphic in $\D$, 
\[
f(A_n)=\diag(f(J_{n, 1}),\dots, f(J_{n, k_n})),
\]
 where
\[
f(J_{n,j})=\begin{bmatrix}
f(\lambda_{n,j}) & f'(\lambda_{n,j}) & \frac{f''(\lambda_{n,j})}{2} & \dots & \frac{f^{(m_{n, j}-1)}(\lambda_{n,j})}{(m_{n, j}-1)!}\\
0 &f(\lambda_{n, j}) & f'(\lambda_{n,j}) & \dots & \frac{f^{(m_{n, j}-2)}(\lambda_{n,j})}{(m_{n, j}-2)!}\\
\vdots & \vdots & \vdots & \vdots & \vdots\\
0 & 0 & \dots & f(\lambda_{n, j}) & f'(\lambda_{n,j})\\
0 & 0 & 0 & 0 & f(\lambda_{n,j})
\end{bmatrix},
\]
one realizes that a holomorphic function vanishes at $A_n$ if and only if it vanishes at its eigenvalues with the right multiplicity.  Specifically, the multiplicity of $\lambda_{n, j}$ as a zero of $f$ must be the maximal size of a Jordan block of $A_n$ associated to $\lambda_{n, j}$. In particular, any function in $\h^2$ that vanishes at $A_n$ is a multiple of the Blaschke product and the function in  \eqref{eqn:Bn} can be re-written as
\[
B_n=\prod_{j=1}^{k_n}b_{\lambda_{n, j}}^{m_{n, j}}.
\]
Hence, for any $n$ in $\N$, the subspace
\begin{equation}
\label{eqn:matmod}
\begin{split}
H_n:=&\h^2\ominus\{f\in\h^2\,|\, f(A_n)=0\}\\
=&\spa\left\{s_{\lambda_{n, j}}, \frac{\partial~s_{\lambda_{n, j}}}{\partial\overline{w}},\dots, \frac{\partial^{m_{n, j}-1}~s_{\lambda_{n, j}}}{\partial\overline{w}^{m_{n, j}-1}}\,\bigg|\, j=1, \dots, k_n\right\}
\end{split}
\end{equation}
containing all the interpolation information of the matrix $A_n$ is in fact a model space:
\[
H_n=H_{B_n}\qquad n\in\N.
\]
Each $H_n$ can be seen also as a \emph{kernel at the matrix $A_n$}. More precisely, let $M$ be a $m\times m$ square matrix with eigenvalues in the unit disc, and let $H_M$ be the associated model space in $\h^2$. Let us define, for any $u$ and $v$ in $\C^m$, the $\h^2$ function
\[
K_M(u, v)(z):=\sum_{n\in\N}\braket{v, M^nu}_{\C^m}z^n\qquad z\in\D.
\]
Then, thanks to the definition of the inner product in $\h^2$,
\begin{equation}
\label{eqn:repmat}
\braket{f, K_{M}(u, v)}_{\h^2}=\braket{f(M)u, v}_{\C^m}\qquad f\in\h^2.
\end{equation}
Equation \ref{eqn:repmat} works as a \emph{reproducing property} for the collection
\[
X_M:=\{K_M(u, v)\,|\, u, v \in\C^m\}.
\]
In particular, since $f(M)=0$ if and only if the right hand side of \eqref{eqn:repmat} vanishes for any $u$ and $v$, we have that $X_M$ coincides with the model space $H_M$. Moreover, \eqref{eqn:repmat} implies that the collection of function in $X_M$ is linear in $v$ and conjugate-linear in $u$, and that for any $\phi$ in $\h^\infty$
\begin{equation}
\label{eqn:invariant}
M_\phi^*(K_M(u, v))=K_M(u, \phi(M)^*v)\qquad u, v\in\C^m,
\end{equation}
extending \eqref{eqn:eigen} to this matrix setting.
Another analogy with the scalar case comes from separation: thanks to \eqref{eqn:invariant} and the commutant lifting Theorem, \cite[case $s=1$ of Corollary 10.30]{john}, \eqref{eqn:anglespick} extends by saying that $A$ is {\bf weakly separated} if the sequence $(H_n)_{n\in\N}$ is weakly separated or, equivalently, if there exists a positive $M$ such that, for any pair of distinct positive integers $n$ and $j$, there exists a bounded analytic function $\phi_{n, j}$ whose $\h^\infty$ norm doesn't exceed $M$ and that separates $A_n$ and $A_j$, that is, 
\[
\phi_{n, j}(A_n)=Id\qquad\phi_{n, j}(A_j)=0.
\]
Following the same idea, if $(H_n)_{n\in\N}$ is strongly separated we say that $A$ is {\bf strongly separated}, and the commutant lifting Theorem, together with \eqref{eqn:invariant}, makes it equivalent to asserting the existence of a bounded sequence $(\phi_n)_{n\in\N}$ in $\h^\infty$ that separates each $A_n$ with the rest of the sequence, i.e.,
\[
\phi_n(A_j)=\delta_{n, j}~Id.
\]
The scalar case has an even more geometric viewpoint on separation via bounded analytic functions: given two points $z$ and $w$ in the unit disc there exists a function $\phi$ whose $\h^\infty$ norm doesn't exceed $M$ that separates $z$ and $w$ (and hence the sine of the angle between $s_z$ and $s_w$ is bounded below by $1/M$) if and only if their \emph{pseudo-hyperbolic distance}
\[
\rho(z, w):=|b_w(z)|
\]
 is bounded below by $1/M$. This extends to the matrix case by looking at the action of the adjoint of the multiplication by a Blaschke product on different model spaces:
\begin{lemma}
\label{lemma:pseudo}
Let $A_1$ and $A_2$ be two square matrices corresponding to the Blaschke products $B_1$ and $B_2$, and let $H_1$ and $H_2$ be the associated model spaces. Then the sine of the angle between $H_1$ and $H_2$ is equal to $\delta>0$ if and only if the restriction of $M^*_{B_1}$ to $H_2$ is bounded below by $\delta$, that is,
\[
\inf_{x\in H_2}||M_{B_1}^*(\hat{x})||=\sin(H_1, H_2).
\]
\end{lemma}
\begin{proof}
Since the orthogonal projection onto $B_1\h^2=\h^2\ominus H_{1}$ is $M_{B_1}M^*_{B_1}$, one gets
\[
\sin(H_1, H_2)=\inf_{x\in H_2\setminus\{0\}}||M_{B_1}M_{B_1}^*(\hat{x})||=\inf_{x\in H_2\setminus\{0\}}||M_{B_1}^*(\hat{x})||,
\]
since $M_{B_1}$ is an isometry.
\end{proof}
The author extended in \cite{intmat} Carleson's characterizations, \cite{carlocm} and \cite{carloss}, of interpolating sequences to sequences of square matrices, together with the characterization of interpolating sequences in terms of Riesz systems conditions from \cite{shapiro}. Nevertheless, the analogous of Theorem \ref{theo:carlo} was proven with the additional assumption that the dimensions of the matrices in $A$ are uniformly bounded, and used the solution of the Feichtinger conjecture:
\begin{theo}
\label{theo:intmat}
 Let $A=(A_n)_{n\in\N}$ be a sequence of matrices with spectra in the unit disc, and let $H=(H_n)_{n\in\N}$ be the associated sequence of model spaces defined in \eqref{eqn:matmod}. The following are equivalent:
\begin{description}
\item[(i)] $A$ is interpolating for $\h^\infty$;
\item[(ii)] $A$ is strongly separated;
\item[(iii)] $H$ is a Riesz system.
\end{description}
Moreover, if the dimensions of the matrices in $A$ are uniformly bounded, (i) (and hence also all the conditions above) is equivalent to 
\begin{description}
\item[(iv)] $A$ is weakly separated and $H$ is a Bessel system.
\end{description}
\end{theo}
Thanks to equivalence between conditions (i), (ii) and (iii) in Theorem \ref{theo:intmat}, Theorem \ref{theo:main} applied to the sequence of model spaces $(H_n)_{n\in\N}$ says, together with Theorem \ref{theo:nik}, that the extra assumption on the dimensions of the matrices in $A$ in condition (iv) of Theorem \ref{theo:intmat} can be dropped. Therefore Theorem \ref{theo:carlo} extends to sequences of matrices of any sizes:
\begin{theo}
$A$ is interpolating if and only if it is weakly separated and $H$ is a Bessel system.
\end{theo}
 In particular, Theorem \ref{theo:carlomat} follows from Theorem \ref{theo:main}.\\

We present below two examples of sequences of matrices having equidistributed eigenvalues. Section \ref{sec:ex} will give an example of an interpolating sequence of matrices, together with some useful tool to estimate the angle between model spaces arising from Blaschke products. Section \ref{sec:feicmodel} will use a similar construction in order to exhibit a Bessel system of model spaces which is not the disjoint union of finitely many weakly separated sequences, and hence not the finite union of finitely many Riesz systems. This will give a negative answer to a question posed by the author in \cite{intmat}, where it was asked whether the positive answer to the Feichtinger conjecture can be extended to multi-dimensional model sub-spaces of the Hardy space.
\subsection{Interpolating Matrices with Equidistributed Eigenvalues}
\label{sec:ex}
Let, for any positive integer $n$, 
\[
\omega_n:=e^\frac{2\pi i}{2^n}
\]
 be a primitive $2^n$-root of unity, and let
\[
W_n:=\diag(1, \omega_n, \dots, \omega_n^{2^n-1})
\]
be a $2^n\times2^n$ diagonal matrix having $2^n$ equi-distributed points on the unit circle as its eigenvalues. Let $(r_n)_{n\in\N}$ be a sequence in $(0, 1)$ that re-scales the sequence $(W_n)_{n\in\N}$ so that its spectra belong to $\D$:
\begin{equation}
\label{eqn:dyadmat}
A_n:=r_n~W_n,\qquad n\in\N.
\end{equation}
We will discuss here how fast must $(r_n)_{n\in\N}$ go to $1$ in order for $A:=(A_n)_{n\in\N}$ to be interpolating. Thanks to Definition \ref{defi:intmat}, if $(A_n)_{n\in\N}$ is interpolating, then so is the sequence of radii $(r_n)_{n\in\N}$. Moreover, if $A$ is interpolating it is trivially a \emph{zero sequence}, that is, there exists bounded analytic function on $\D$ that vanishes on $A$ and that doesn't vanish outside the spectra of the matrices in $A$. It turns out that those two conditions are enough to characterize interpolating sequences of matrices that look like \eqref{eqn:dyadmat}:
\begin{theo}
\label{theo:example}
The sequence of matrices defined in \eqref{eqn:dyadmat} is interpolating if and only if it is a zero sequence and $(r_n)_{n\in\N}$ is interpolating. 
\end{theo}
\begin{rem}
\label{rem:weakint}
Since $(r_n)_{n\in\N}$ approaches the unit circle radially, asking that it is interpolating is actually the same as asking that it is just weakly separated \cite{treatise}[Lec. X, Cor. 5].
\end{rem}
The proof of Theorem \ref{theo:example} requires that we are able to estimate (from below) the angle between two model spaces arising from Blaschke products. Such a tool is the content of Lemma \ref{lemma:blaschke} below. Let $(B_n)_{n\in\N}$ be a sequence of Blaschke products such that $B=\prod_{n\in\N}B_n$ converges uniformly on any compact subset of $\D$ to a non zero inner function, and let $(H_n)_{n\in\N}$ be the associated sequence of model spaces in $\h^2$.  For any subset $\sigma$ of $\N$ we will define, for the sake of brevity,
\[
H_\sigma:=\overline{\underset{i\in\sigma}{\spa}}\{H_i\}
\]
and 
\[
B_\sigma=\prod_{i\in\sigma} B_i.
\]
\begin{lemma}
\label{lemma:blaschke}
Let $\sigma$ and $\tau$ be two disjoint subsets of $\N$, and suppose that $(H_i)_{i\in\sigma}$ is a Riesz system with Riesz bound $\gamma$. Then
\[
\sin(H_\sigma, H_\tau)\geq\frac{1}{\gamma^2}~\inf_{i\in\sigma}\sin(H_i, H_{\tau}).
\]
\end{lemma}
\begin{proof}
For any $i$ in $\sigma$ let $\delta_i:=\sin(H_i, H_\tau)$, and let $\delta:=\inf_{i\in\sigma}\delta_i$. It suffices to show that 
\[
T\colon H_{\sigma\cup\tau}\to H_{\sigma\cup\tau}
\]
such that
\[
T_{|H_\sigma}=\delta~Id_{|H_\sigma}\qquad T_{|H_\tau}=0
\]
is bounded by $\gamma^2$. Let $T_i$ be, for any $i$ in $\sigma$, the restriction to $H_i$ of $M^*_{B_\tau}$. Thanks to Lemma \ref{lemma:pseudo}, each $T_i$ is bounded below by $\delta_i$. Fix then a vector $x=u+v$ in $H_{\sigma\cup\tau}$, where $u$ is in $H_\sigma$ and $v$ is in $H_\tau$. There exists a sequence $(h_i)_{i\in\sigma}$ of unit vectors such that $h_i$ is in $H_i$ for any $i$ in $\sigma$ so that $u$ can be written as a linear combination of $(h_i)_{i\in\sigma}$: 
\[
u=\sum_{i\in\sigma}\alpha_i h_i.
\]
Since each $T_i$ is a contraction and it is bounded below by $\delta$, the sequence $(T_i(h_i))_{i\in\N}$ is bounded above and below. Moreover, $T_i(h_i)$ belongs to $H_i$, since each $H_i$ is invariant under $M_{B_\tau}$, and therefore
\[
\begin{split}
||T(x)||^2&=||T(u)||^2=\left|\left|\sum_{i\in\sigma}\delta\alpha_i h_i\right|\right|^2\\
&\leq\gamma^2\delta^2\sum_{i\in\sigma}|\alpha_i|^2\leq\gamma^2\sum_{i\in\sigma}\delta_i^2|\alpha_i|^2\\
&\leq\gamma^2\sum_{i\in\sigma}|\alpha_i|^2||T_i(h_i)||^2\leq \gamma^4\left|\left|\sum_{i\in\sigma}\alpha_iT_i(h_i) \right|\right|^2\\
&=\gamma^4||M^*_{B_\tau}(u)||^2=\gamma^4||M^*_{B_\tau}(x)||^2.
\end{split}
\]
Since $M^*_{B_\tau}$ is a contraction, this shows that the norm of $T$ doesn't exceed $\gamma^2$, as we claimed.
\end{proof}
\begin{rem}
\label{rem:sin}
Suppose that also $(H_j)_{j\in\tau}$ is a Riesz system. Then a double application of Lemma $\ref{lemma:blaschke}$ implies that the distance between $H_\sigma$ and $H_\tau$ is comparable with the minimal distance attained by a model space labeled by an index in $\sigma$ and one with a label in $\tau$. If the sequence $(H_n)_{n\in\N}$ is weakly separated, this says roughly speaking that the set of \emph{sparse subsequences} of $(H_n)_{n\in\N}$ is a separated set as well.
\end{rem}
We are now ready to prove Theorem \ref{theo:example}. Here $(B_n)_{n\in\N}$ and $(H_n)_{n\in\N}$ arise from the sequence of matrices $A$ defined in \eqref{eqn:dyadmat}. We need to show that $A$ is interpolating, provided that it is a zero sequence and that $(r_n)_{n\in\N}$ is weakly separated, thanks to Remark \ref{rem:weakint}.
\begin{proof}[Proof of Theorem \ref{theo:example}]
 Let, for any positive integer $n$, 
\begin{equation}
\label{eqn:alpha}
r_n=1-\alpha_n2^{-n}.
\end{equation}
 Since $A$ is a zero sequence, then the spectra of the matrices in $A$ form a zero sequence and therefore
\[
\sum_{n\in\N}\alpha_n<\infty.
\]
Let $\gamma_n$ be the Riesz bound of the basis $\{\hat{s}_{r_n},\dots, \hat{s}_{r_n\omega_n^{2^n-1}}\}$ of $H_n$. Then $(\gamma_n)_{n\in\N}$ is uniformly bounded if and only if the strong separation constants
\[
\prod_{l=1}^{2^n-1}|b_{r_n\omega_n^l}(r_n)|^2
\]
are uniformly bounded below. Let, for any $j$ and $n$ in $\N$, 
\begin{equation}
\label{eqn:mjn}
\begin{split}
M_j(n):=&\sum_{l=1}^{2^j}|\braket{\hat{s}_{r_j\omega_j^l}, \hat{s}_{r_n}}|^2\\
=&\sum_{l=1}^{2^j}\frac{(1-r_n^2)(1-r_j^2)}{|1-r_nr_j\omega_j^l|^2}.
\end{split}
\end{equation}
Since $(\alpha_n)_{n\in\N}$ is bounded, then the  weak separation constant of the set $\{r_n, \dots, r_n\omega^{2^ n-1}\}$ is uniformly bounded below in $n$. Since also 
\[
|b_w(z)|^2=1-|\braket{\hat{s}_w, \hat{s}_z}|^2\qquad w, z\in\D, 
\]
then $(\gamma_n)_{n\in\N}$ is bounded if and only if $(M_n(n))_{n\in\N}$ is, which is the case, thanks to Lemma \ref{lemma:mjn} below, because $(\alpha_n)_{n\in\N}$ is bounded. We want to show that $(H_n)_{n\in\N}$ is strongly separated:
\[
\inf_{n\in\N}\sin\left(H_n,\underset{j\ne n}{\overline{\spa}}\{H_j\}\right)>0.
\]
Since $(\gamma_n)_{n\in\N}$ is bounded, thanks to Lemma \ref{lemma:blaschke} the sine between $H_n$ and $\underset{j\ne n}{\overline{\spa}}\{H_j\}$ is comparable
 to  
\[
\sin\left(\hat{s}_{r_n}, \underset{j\ne n}{\overline{\spa}}\{H_j\}\right)=\prod_{j\ne n}|B_j(r_n)|.
\]
 In fact, the sine between $H_n$ and $\underset{j\ne n}{\overline{\spa}}\{H_j\}$ is comparable to 
\[
\underset{l=0, \dots, 2^n-1}{\min}\sin\left(s_{r_n\omega_n^l}, \underset{j\ne n}{\overline{\spa}}\{H_j\}\right)=\underset{l=0, \dots, 2^n-1}{\min}|B_j(r_n\omega_n^l)|.
\] 
Nevertheless, the minimum is attained at $l=0$, since
\[
B_j(z)=\frac{r_j^{2^j}-z^{2^j}}{1-r_j^{2^j}z^{2^{j}}}=b_{r_j^{2^j}}(z^{2^j})
\]
and the point on the circle of radius $r_n$ which is closest (with respect the pseudo-hyperbolic distance) to $r_j^{2^j}$ is precisely $r_n$. \\
We are then left to show that
\begin{equation}
\label{eqn:newref}
\inf_{n\in\N}\prod_{j\ne n}|B_j(r_n)|>0.
\end{equation}
This is true if and only if each term on the product is uniformly bounded below and 
\[
\sup_{n\in\N}\sum_{j\in\N}M_j(n)<\infty.
\]
Such sum converges thanks to Lemma \ref{lemma:mjn}, since $(\alpha_n)_{n\in\N}$ is summable. Moreover, each term on the product in \ref{eqn:newref} is uniformly bounded below, since $(r_n)_{n\in\N}$ is weakly separated and since, thanks to Lemma \ref{lemma:blaschke}, the pseudo-hyperbolic distance between $r_n$ and $r_j$ is, for any $n\ne j$, comparable to $\sin(\hat{s}_{r_n}, H_j)=|B_j(r_n)|$.
\end{proof}
A technical tool for the proof of Theorem \ref{theo:example} is the following computation, which relates the quantity $M_j(n)$ to the parameters $(\alpha_n)_{n\in\N}$ defined in \eqref{eqn:alpha}:
\begin{lemma}
\label{lemma:mjn}
Let
\[
r_n:=1-\alpha_n2^{-n}\qquad n\in\N
\]
be a sequence in $(0, 1)$ and let $M_j(n)$ be defined as in \eqref{eqn:mjn}. Then, for any $j$ and $n$ positive integers,
\[
M_j(n)\simeq\frac{\alpha_n\alpha_j}{\alpha_n+\alpha_j2^{n-j}-\alpha_n\alpha_j2^{-j}}.
\]
\end{lemma}
\begin{proof}
Let, for any $j$ in $\N$ and for any $l=-2^{j-1}, \dots, 2^{j-1}$,
\[
\theta^j_l:=\arg{\omega_j^l}=\frac{2\pi l}{2^j}.
\]
Then 
\[
1-\cos(\theta^j_l)\simeq\frac{ (\theta^j_l)^2}{2}\simeq\frac{l^2}{2^{2j+1}}\qquad l=-2^{j-1}, \dots, 2^{j-1}
\]
and therefore
\[
\begin{split}
M_j(n)&\simeq(1-r_j)(1-r_n)\sum_{l=-2^{j-1}}^{2^{j-1}}\frac{1}{|1-r_nr_j\omega_j^l|^2}\\
&=(1-r_j)(1-r_n)\sum_{l=-2^{j-1}}^{2^{j-1}}\frac{1}{(1-r_nr_j)^2+2((1-\cos(\theta^j_l)))r_nr_j}\\
&\simeq\frac{(1-r_j)(1-r_n)}{(1-r_jr_n)^2}\sum_{l=1}^{2^{j-1}}\frac{1}{1+\left(\frac{\sqrt{r_nr_j}}{2^j(1-r_nr_j)}~l\right)^2}\\
&\simeq\frac{(1-r_j)(1-r_n)}{(1-r_jr_n)^2}\int_{1}^{2^{j-1}}\frac{1}{1+\left(\frac{\sqrt{r_nr_j}}{2^j(1-r_nr_j)}~x\right)^2}\,dx\\
&=\frac{2^j(1-r_j)(1-r_n)}{(1-r_jr_n)\sqrt{r_jr_n}}\int_{\frac{\sqrt{r_jr_n}}{2^j(1-r_jr_n)}}^{\frac{\sqrt{r_jr_n}}{(1-r_jr_n)}}\frac{1}{1+x^2}\,dx\\
&\simeq\frac{2^j(1-r_j)(1-r_n)}{1-r_jr_n}\\
&=\frac{\alpha_n\alpha_j}{\alpha_n+\alpha_j2^{n-j}-\alpha_n\alpha_j2^{-j}}.
\end{split}
\]
\end{proof}
\subsection{Bessel Systems of Model Spaces}
\label{sec:feicmodel}
Thanks to the positive answer to the Feichtinger conjecture, any Bessel system of lines in a Hilbert space is the disjoint union of finitely many Riesz systems. We show here that this is not the case for multi-dimensional model spaces in $\h^2$, as we will construct a sequence of matrices $A$ which can not be written as the disjoint union of finitely many weakly separated sequences and whose associated sequence of model spaces is a Bessel system. This implies that \cite[Th. 9.11]{john} doesn't extend to multi-dimensional model spaces, and since any Riesz system is weakly separated it will show that the positive answer to the Feichtinger conjecture doesn't extend to multi-dimensional model spaces whose dimensions are not uniformly bounded.\\
Pick a divergent sequence $(m_n)_{n\in\N}$, and let $(r_n)_{n\in\N}$ be a sequence of radii that converges to $1$, and whose rate of convergence will be determined later. For any such a sequence, there exists a second sequence of radii $(s_n)_{n\in\N}$ such that 
\begin{equation}
\label{eqn:closecirc}
\rho(r_n, s_n)\leq\frac{1}{n+1}\qquad n\in\N.
\end{equation}
 For each $n$ in $\N$, consider $m_n$ diagonal matrices, $A_{n, 1}, \dots, A_{n, m_n}$, each one of size $m_n-1$, and whose eigenvalues are placed as follows: if $\omega_n$ is a $\binom{m_n}{2}$-th primitive root of unity, the set of all eigenvalues of $A_{n, 1}, \dots, A_{n, m_n}$ is 
\begin{equation}
\label{eqn:eigenlist}
\left\{r_n, r_n\omega_n, \dots, r_n\omega_n^{\binom{m_n}{2}-1}\right\}\cup\left\{s_n, s_n\omega_n, \dots, s_n\omega_n^{\binom{m_n}{2}-1}\right\}.
\end{equation}
We now need to assign each of those $m_n(m_n-1)$ eigenvalues to the matrices $A_{n, 1}, \dots, A_{n, m_n}$ so that
\begin{equation}
\label{eqn:notfu}
\sin\left(H_{A_{n, i}}, H_{A_{n, j}}\right)\leq\frac{1}{n+1}\qquad i, j=1, \dots, m_n,
\end{equation}
where $H_{n, i}$ is the model space in $\h^2$ associated to the matrix $A_{n, i}$.
Observe that this would imply that the sequence 
\[
A:=\left\{A_{n, j}\,|\,n\in\N, j=1, \dots, m_n\right\}
\]
can not be written as the disjoint union of finitely many weakly separated sub-sequences, since $(m_n)_{n\in\N}$ is divergent.  More importantly, this would hold for any choice of $(r_n)_{n\in\N}$. In order to achieve \eqref{eqn:notfu}, we assign the eigenvalues in \eqref{eqn:eigenlist} to the matrices $A_{n, 1}, \dots, A_{n, m_n}$ so that the set of non-ordered pairs
\[
\left\{\left(r_n\omega_n^l, s_n\omega_n^l\right)\,|\,l=1, \dots, \binom{m_n}{2}\right\}
\]
coincide with the set of non-ordered pairs 
\[
\{(\lambda,  \gamma)\,|\, \lambda\in\sigma(A_{n, i}), \gamma\in\sigma(A_{n, j}),\, i\ne j\}.
\]
Thus \eqref{eqn:notfu} follows from the fact that, for any $i\ne j=1,\dots, m_n$, there exists an eigenvalue $\lambda$ of $A_{n, i}$ and an eigenvalue $\gamma$ of $A_{n, j}$ so that $$\rho(\lambda, \gamma)\leq\frac{1}{n+1},$$ thanks to \eqref{eqn:closecirc}. 
\begin{ex}
In order to have an explicit example in mind, let $m_1=4$, and let then $\omega_1$ be a primitive $6$-th root of unity. One can then choose
\[
\begin{split}
r_1\in\sigma(A_{1, 1})\quad& s_1\in\sigma(A_{1,  2})\\
r_1\omega\in\sigma(A_{1, 1})\quad&s_1\omega\in\sigma(A_{1, 3})\\
r_1\omega^2\in\sigma(A_{1, 1})\quad&s_1\omega^2\in\sigma(A_{1, 4})\\
r_1\omega^3\in\sigma(A_{1, 2})\quad&s_1\omega^3\in\sigma(A_{1, 3})\\
r_1\omega^4\in\sigma(A_{1, 2})\quad&s_1\omega^4\in\sigma(A_{1, 4})\\
r_1\omega^5\in\sigma(A_{1, 3})\quad&s_1\omega^5\in\sigma(A_{1, 4}).
\end{split}
\]
In general, the idea is to assign, for all $n$, the $m_n(m_n-1)$ points on the two very close circle of radius $r_n$ and $s_n$ in a way so that the $\binom{m_n}{2}$ pairs of  points very close to each other correspond to the set of pairs $\{(A_{n, i}, A_{n, j})\}$.
 \end{ex}
 What is left now to show is that, if $(r_n)_{n\in\N}$ is chosen to be converging to $1$ adequately fast, the sequence of model spaces associated with the sequence $A$ is a Bessel system. Let $A'=(z_j)_{j\in\N}$ be the (scalar) sequence of all the eigenvalues of the matrices in $A$. Observe that we can recursively choose the sequence $(r_n)_{n\in\N}$ to approach $1$ fast enough so that 
\begin{equation}
\label{eqn:recursive}
\sup_{z\in\D}\sum_{j\in\N}1-|b_{z_j}(z)|^2<\infty.
\end{equation}
Indeed, one can use the same idea of \cite{intmat}[Th. 2.2] and add the contribution of the eigenvalues in \eqref{eqn:eigenlist} once at a time in the sum in \eqref{eqn:recursive}. More precisely, let $A_n'$ be the set in \eqref{eqn:eigenlist}, and let
\[
Q_n(z):=\sum_{j=1}^n\sum_{z_l\in A'_j}1-|b_{z_l}(z)|^2.
\]
Then, once we choose $r_1, \dots, r_n$, there exists a radius $t_n$ such that
\[
\sup_{t_n\leq|z|<1}Q_n(z)\leq\frac{1}{2^n},
\]
and it suffices then to choose $r_{n+1}$ so that
\[
\sum_{z_l\in A'_{n+1}}1-|b_{z_l}(z)|^2\leq\frac{1}{2^n}\qquad |z|\leq t_n
\]
and so that the Riesz bound of the set of normalized kernels at the points in $A'_{n+1}\cap\{|z|=r_{n+1}\}$ is uniformly bounded (in fact, since $m_{n+1}$ is fixed one can make such Riesz bound arbitrarily close to $1$, by choosing $r_{n+1}$ close enough to $1$ ).\\
Thanks to \eqref{eqn:recursive} and Remark \ref{rem:bessel:lines} the sequence of lines spanned by the Szegö kernels at the point of $A'$ is a Bessel system. This, together with an extra separation condition on the eigenvalues of the matrices in $A$, implies that the model spaces associated with the sequence $A$ forms a Bessel system:
\begin{lemma}
\label{lemma:regroup}
Let $(H_n)_{n\in\N}$ be a sequence of closed sub-spaces of a Hilbert space $\mathcal{H}$, and let, for any $n$ in $\N$, $\{x_n^1, \dots, x_n^{m_n}\}$ be a basis of $H_n$ made of unit vectors such that
\begin{equation}
\label{eqn:boundbelow}
\sum_{l=1}^{m_n}|c_l|^2\leq C_n^2~\left|\left|\sum_{l=1}^{m_n}c_lx_n^l\right|\right|^2, \qquad c_1, \dots, c_{m_n}\in\C.
\end{equation}
If 
\begin{equation}
\label{eqn:unifboundbelow}
C:=\sup_{n\in\N}C_n<\infty
\end{equation}
 and $(x_n^l)$ is a Bessel system with bound $M$, then $(H_n)_{n\in\N}$ is a Bessel system with bound $CM$.
\end{lemma}
\begin{proof}
Let $(h_n)_{n\in\N}$ be a sequence of unit vectors in $\mathcal{H}$ such that $h_n$ belongs to $H_n$ for any $n$ in $\N$, and let $(a_n)_{n\in\N}$ be an $l^2$ sequence. Write $h_n=\sum_{l=1}^{m_n}b_n^lx_n^l$, and observe that 
\[
\begin{split}
\left|\left|\sum_{n\in\N}a_nh_n\right|\right|^2=&\left|\left|\sum_{n\in\N}\sum_{l=1}^{m_n}a_nb_n^lx_n^l\right|\right|^2\\
\leq& M^2\sum_{n\in\N}|a_n|^2\sum_{l=1}^{m_n}|b_n^l|^2\\
\leq& C^2M^2\sum_{n\in\N}|a_n|^2,
\end{split}
\]
thanks to \eqref{eqn:boundbelow} and \eqref{eqn:unifboundbelow}.
\end{proof}
Since the sequence $(m_n)_{n\in\N}$ is fixed, by eventually increasing the rate of convergence of $(r_n)_{n\in\N}$ to $1$ we can assume that the Reisz bound of the kernels basis of each $H_{A_n}$ is uniformly bounded in $n$, thus in particular condition \eqref{eqn:boundbelow} holds. Indeed, although the eigenvalues of a given matrix $A_n$ might belong to two circle very close to each other, our construction ensures that their arguments are separated by at least an angle of $\frac{2\pi}{\binom{m_n}{2}}$. Therefore, thanks to Lemma \ref{lemma:regroup}, the sequence of model spaces associated to the sequence $A$ is a Bessel system. 

\end{document}